\documentclass[11pt]{amsart}
 
 \usepackage[a4paper, total={6in, 8in}]{geometry}
\usepackage{comment}
\usepackage{soul}
\usepackage{xfrac}
\usepackage{bm}
\usepackage[dvipsnames]{xcolor}
\usepackage[normalem]{ulem}

\newcommand{\da}[1]{{\color{black}{ #1}}}

\makeatletter

\numberwithin{equation}{section}
\numberwithin{figure}{section}

\begin{document}
\newtheorem{theorem}{Theorem}[section] \newtheorem{lemma}[theorem]{Lemma}
\newtheorem{definition}[theorem]{Definition} \newtheorem{example}[theorem]{Example}
\newtheorem{proposition}[theorem]{Proposition} \newtheorem{corollary}[theorem]{Corollary}
\newtheorem{conjecture}[theorem]{Conjecture} 
\newtheorem{problem}[theorem]{Problem}
\newtheorem{question}[theorem]{Question}
\newtheorem{remark}{Remark} 
\newtheorem{main}{Main Theorem}
\global\long\def\wta{{\rm {wt} } a }

 \global\long\def\R{\frak{R}}
\global\long\def\<{\langle}
\newcommand{\bea}{\begin{eqnarray}}
\newcommand{\eea}{\end{eqnarray}}
\newcommand{\wh}{ {\bf w}_{ \bm{ \lambda}, \bm{ \mu} }}
\newcommand{\whf}{ {\bf w}_{ \bm{ \alpha}, \bm{ \beta} }}
\newcommand{\whab}{ {\bf w}_{ \bm{ a}, \bm{ b} }}
 \global\long\def\LL{\mathcal{L}}
 \global\long\def\>{\rangle}
 \global\long\def\t{\tau}
 \global\long\def\a{\alpha}
 \global\long\def\e{\epsilon}
 \global\long\def\l{\lambda}
 \global\long\def\ga{\gamma}
 \global\long\def\L{ L}
 \global\long\def\b{\beta}
 \global\long\def\om{\omega}
 \global\long\def\o{\omega}
 \global\long\def\c{\chi}
 \global\long\def\ch{\chi}
 \global\long\def\cg{\chi_{g}}
 \global\long\def\ag{\alpha_{g}}
 \global\long\def\ah{\alpha_{h}}
 \global\long\def\ph{\psi_{h}}
 \global\long\def\gi{\gamma}
 \global\long\def\supp{{\rm {supp}}}
 \global\long\def\GG{\mathcal{G}}
 \global\long\def\NN{\mathcal{N}}
 \global\long\def\wtb{{\rm {wt} } b }
   \global\long\def\ee{{equation}}
 \global\long\def\g{\frak{g}}
 \global\long\def\tg{\tilde{\frak{g}} }
 \global\long\def\hg{\hat{\frak{g}} }
 \global\long\def\hb{\hat{\frak{b}} }
 \global\long\def\hn{\hat{\frak{n}} }
 \global\long\def\h{\frak{h}}
 \global\long\def\wt{{\rm {wt} } }
 \global\long\def\V{{\cal V}}
 \global\long\def\hh{\hat{\frak{h}} }
 \global\long\def\n{\frak{n}}
 \global\long\def\Z{\mathbb{Z}}
 \global\long\def\lar{\longrightarrow}
 \global\long\def\X{\mathfrak{X}}

\global\long\def\Zp{{\Bbb Z}_{\ge0} }

\global\long\def\N{\mathbb{N}}
 \global\long\def\C{\mathbb{C}}
 \global\long\def\Q{\Bbb Q}
 \global\long\def\WW{\boldsymbol{\mathcal{W}}}
\global\long\def\gl{\mathfrak{gl}}
 \global\long\def\sl{\mathfrak{sl}}
 \global\long\def\la{\langle}
 \global\long\def\ra{\rangle}
 \global\long\def\bb{\mathfrak{b}}
 \global\long\def\triplet{\mathcal{W}(p)}
 \global\long\def\striplet{\mathcal{SW}(m)}

\global\long\def\nordplus{\mbox{\scriptsize\ensuremath{{+\atop +}}}}
 \global\long\def\ds{{\displaystyle }}
\global\long\def\halmos{\rule{1ex}{1.4ex}}
 \global\long\def\pfbox{\hspace*{\fill}\mbox{\ensuremath{\halmos}}}
\global\long\def\epfv{\hspace*{\fill}\mbox{\ensuremath{\halmos}}}
 \global\long\def\nb{{\mbox{\tiny{ \ensuremath{{\bullet\atop \bullet}}}}}}
\global\long\def\nn{\nonumber}

\global\long\def\HH{\widetilde{\mathcal{T}}}
 \global\long\def\vak{{\bf 1}}

\global\long\def\sect#1{\section{#1}\setcounter{equation}{0}\setcounter{rema}{0}}
 \global\long\def\ssect#1{\subsection{#1}}
 \global\long\def\sssect#1{\subsubsection{#1}}


\title[Whittaker modules for  $\widehat{\mathfrak gl}$ and   $\mathcal W_{1+ \infty}$]{Whittaker modules for  $\widehat{\mathfrak gl}$ and   $\mathcal W_{1+ \infty}$--modules  which are not tensor products }

\author{Dra\v{z}en Adamovi\'c,  Veronika Pedi\' c Tomi\' c  }

 \keywords{vertex algebra, Whittaker modules,   $\mathcal W_{1+ \infty}$, typical modules }
\subjclass[2010]{Primary    17B69; Secondary 17B20, 17B65}

\begin{abstract}
We consider the Whittaker modules $M_{1}(\bm{\lambda},\bm{\mu})$  for the Weyl vertex algebra $M$ (also called $\beta \gamma$ vertex algebra), constructed in \cite{ALPY}, where it was proved that these modules are irreducible for each finite cyclic orbifold $M^{\Z_n}$. In this paper, we consider the modules $M_{1}(\bm{\lambda},\bm{\mu})$ as modules for the ${\Z}$--orbifold of $M$, denoted by $M^0$. $M^0$   is isomorphic to the vertex algebra
$\mathcal W_{1+\infty, c=-1} =   \mathcal M(2) \otimes M_1(1)$ which is the  tensor product of the Heisenberg vertex algebra $M_1(1)$ and the singlet algebra $\mathcal M(2)$   (cf. \cite{A-singlet},  \cite{KR}, \cite{Wa}). Furthermore, these modules are also modules 
of  the Lie algebra  $\widehat{\mathfrak gl}$ with central charge $c=-1$. We prove that  they are reducible as $\widehat{\mathfrak gl}$--modules (and therefore also as $M^0$--modules), and we completely describe their irreducible quotients
$L(d,\bm{\lambda},\bm{\mu})$.  
We show that $L(d,\bm{\lambda},\bm{\mu})$ in   most cases are  not tensor product modules for the vertex algebra $ \mathcal M(2) \otimes M_1(1)$. Moreover, we show that all constructed modules are typical  in the sense that they are irreducible  for the Heisenberg-Virasoro vertex subalgebra of $\mathcal W_{1+\infty, c=-1}$.
\end{abstract}
\maketitle
\tableofcontents

\section{Introduction}
 
  
The representation theory of Weyl algebras provides an important tool in the construction of modules for finite and infinite-dimensional Lie algebras, and vertex algebras. Let $\widehat{\mathcal A}$ denote the Weyl algebra, and let $M$ denote the associated Weyl vertex algebra ($\beta \gamma$--system) of rank one. In \cite{ALPY}, we have proved that irreducible Whittaker modules $M_{1}(\bm{\lambda},\bm{\mu})$ 
for the Weyl vertex algebra $M$ remain irreducible when we restrict them to the ${\Z}_n$--invariant subalgebra $M ^{\Z_n}$. In this paper, we consider a limit case and restrict these modules to a ${\Z}$--invariant subalgebra $M^0$. Surprisingly, we prove that the irreducible Weyl modules of Whittaker type are always reducible as $M^0$--modules

 However, the analysis of these reducible modules leads to many interesting observations and new constructions. 

First, we recall that the $\Z$--invariant subalgebra $M^0$ of the rank one Weyl vertex algebra $M$ is isomorphic to the simple  vertex algebra
 $\mathcal W_{1+\infty, c}$ 
 at central charge $c=-1$ (cf. \cite{KR},  \cite{Wa}).
  Thus every Whittaker module for the Weyl vertex algebra is a module for $\mathcal W_{1+\infty, c=-1}$.
On the other hand, the representation theory of the vertex algebra $\mathcal W_{1+\infty, c}$ is deeply connected with the representation theory of the Lie algebra $\widehat{\mathfrak gl}$, which is a central extension of the Lie algebra of infinite matrices (cf. \cite{KR0}, \cite{KR}). More precisely, the universal  vertex algebra $\mathcal W^c _{1+\infty}$  is 
the vacuum module of central charge $c$  for the Lie algebra  $\widehat{\mathcal D}$, which is the central extension of the Lie algebra of complex regular differential operators on ${\C}^{*}$. Moreover, there exists a homomorphism $\Phi_0 : \widehat{\mathcal D} \rightarrow \widehat{\mathfrak gl}$, defined by (\ref{formula-D}); hence each $\widehat{\mathfrak gl}$--module of central charge $c$ is a $\widehat{\mathcal D}$--module  of the same central charge. If the resulting module is restricted, we get a $\mathcal W^c _{1+\infty}$--module.  Using the Weyl (vertex) algebras, one can construct 
 $\widehat{\mathfrak gl}$--modules which become modules for the simple vertex algebra $\mathcal W_{1+\infty, c}$.
Given the Lie homomorphism  $\widehat{\mathfrak gl} \rightarrow \widehat{\mathcal A}$ defined by 
 \bea  E_{i,j} \mapsto :a(-i) a^*(j):,   \quad 
C\mapsto  -1, \nonumber
\eea
every restricted  $\widehat{\mathcal A}$--module has the structure of a $\widehat{\mathfrak gl}$--module, and  one shows that they are also $\mathcal W_{1+\infty, c=-1}$--modules.

Similarly as in the case of the quasi-finite modules from \cite{KR}, we show in Theorem  \ref{structure-equivalence} that the structure of the module $M_{1}(\bm{\lambda},\bm{\mu})$ as a
 $\mathcal W_{1+\infty, c=-1}$--module is equivalent to the structure of  the same module as a $\widehat{\mathfrak gl}$--module.


As far as we have noticed, Whittaker modules for the Lie algebra $\widehat{\mathfrak gl}$ and for the vertex algebra $\mathcal W_{1+\infty}$ haven't been investigated in the literature yet.

 Hence our main task is to analyze $M_{1}(\bm{\lambda},\bm{\mu})$  as a $\widehat{\mathfrak gl}$--module. This will directly imply the structure of $M_{1}(\bm{\lambda},\bm{\mu})$ as a $\mathcal W_{1+\infty}$--module.
 We prove:
 
 \begin{itemize}
 \item  $M_{1}(\bm{\lambda},\bm{\mu})$ is a cyclic $\widehat{\mathfrak gl}$--module of Whittaker type for the Whittaker pair $(\widehat{\mathfrak gl}, \mathfrak p)$, where $\mathfrak p$ is a certain commutative subalgebra of upper triangular matrices of  
 $\widehat{\mathfrak gl}$. Thus, $M_{1}(\bm{\lambda},\bm{\mu})$ is not a classical Whittaker module for $\widehat{\mathfrak gl}$, but a generalized Whittaker module in the sense of \cite{BM}.
 
 \item In order to see that $M_{1}(\bm{\lambda},\bm{\mu})$ is not irreducible, we use the Casimir element $I = \sum_{i} E_{i,i}$ of
  $U( \widehat{\mathfrak gl})$. We show that $I$ acts non-trivially on $M_{1}(\bm{\lambda},\bm{\mu})$.
 
 \item We prove that ${\C}[I]{\wh}$ provides all
 Whittaker vectors in $M_{1}(\bm{\lambda},\bm{\mu})$ and that for every $d \in {\C}$, the quotient
 $$ L(d, \bm{\lambda},\bm{\mu}) =  \frac{ M_{1}(\bm{\lambda},\bm{\mu})}{ \langle (I-d) {\C}[I] {\wh}  \rangle }$$
 is simple  both as a $\widehat{\mathfrak gl}$--module and as $\mathcal W_{1+\infty,c=-1}$--module.
 \end{itemize}
 
 \subsection*{Free field realization of  $L(d, \bm{\lambda},\bm{\mu})$.} Note that  modules $L(d, \bm{\lambda},\bm{\mu})$ are constructed as quotients of Whittaker Weyl modules. A natural question is to provide an explicit realization of these simple modules.

 There exist another realization of Whittaker 
 $\mathcal W_{1+\infty,c}$--modules which are based on the fact that  $\mathcal W_{1+\infty,c}$ can be embedded into Heisenberg vertex algebra for integral central charges. Let $M_n(1)$ denotes the Heisenberg vertex algebra generated by $n$ Heisenberg fields.
 \begin{itemize}
 \item  $\mathcal W_{1+\infty,c=1}$ is isomorphic to $M_1(1)$. 
 \item $\mathcal W_{1+\infty,c=n}$ is isomorphic to  the principal W-algebra $\mathcal W_1(\mathfrak{gl}(n), f_{princ})$ which is a vertex subalgebra of  $M_n(1)$.  (cf. \cite{FKRW}).
 \item $\mathcal W_{1+\infty,c=-1}$ is isomorphic to the tensor product $\mathcal M(2) \otimes M_1(1)$, where $\mathcal M(2)$ is the singlet vertex algebra  at central charge $-2$ (isomorphic to the simple  $\mathcal W(2,3)$-algebra at $c=-2$) and $M_1(1)$ is rank one Heisenberg vertex algebra (cf. \cite{Wa}). In particular, $\mathcal W_{1+\infty,c=-1}$ is embedded into $M_2(1)$.
  \item In general, $\mathcal W_{1+\infty,c=-n}$ is embedded into  $M_{2n}(1)$ (cf. \cite{A-2001}).  
 \end{itemize}
 To the best of our knowledge, the following problem is still unsolved:
 \begin{problem} For embedding $\mathcal W_{1+\infty,c} \hookrightarrow M_n(1)$ mentioned above, identify Whittaker $M_n(1)$--modules as Whittaker $\mathcal W_{1+\infty,c}$--module.
 \end{problem}
 
 We will be focused on the case $c=-1$. Then $\mathcal W_{1+\infty,c=-1}$ is a subalgebra of $M_2(1)$ and every Whittaker $M_2(1)$--module is naturally Whittaker $\mathcal W_{1+\infty,c=-1}$--module. Then by using the isomorphism
 $\mathcal W_{1+\infty,c=-1} \cong \mathcal M(2) \otimes M_1(1)$  one can realize a family of Whittaker modules as the tensor product modules
 \bea
 && Z_1 \otimes Z_2 \label{form-decomp} \eea
  where  $Z_1$ is an  irreducible Whittaker $\mathcal M(2)$--module and $Z_2$ is an irreducible Whittaker $M_1(1)$--module. K. Tanabe showed in \cite{T2} that every  Whittaker $\mathcal M(p)$--module is obtained from Whittaker modules for the Heisenberg vertex algebra. Therefore, we have a family of irreducible $\mathcal W_{1+\infty,c=-1}$--modules obtained using free-field realization from the  Whittaker $M_2(1)$--modules.
 \begin{question}
Can we realize $L(d, \bm{\lambda},\bm{\mu})$  as irreducible Whittaker $M_2(1)$--modules? This is equivalent to the question can we represent $L(d, \bm{\lambda},\bm{\mu})$ as the tensor product modules (\ref{form-decomp})?
 \end{question}

  In Section  \ref{sect-bosonic}  we find a direct realization of   $L(d, \bm{\lambda},\bm{\mu})$ as $M_2(1)$--modules in the special cases when  $M_{1}(\bm{\lambda},\bm{\mu})$ has the structure of a module for the vertex algebra $\Pi(0)$ (cf. Proposition \ref{identification-1}). Then the Casimir element $I$ can be seen as an element of the vertex algebra $\Pi(0)$.  
  
  But in general $L(d, \bm{\lambda},\bm{\mu})$ does not have form (\ref{form-decomp}). The reason is that   $L(d, \bm{\lambda},\bm{\mu})$ contains vectors which are not locally finite for the action the element $J^0(k)$, for $k>0$,  which is a necessary condition for a $\mathcal W_{1+\infty,c=-1}$--module to have the form (\ref{form-decomp}).
  Therefore modules $L(d, \bm{\lambda},\bm{\mu})$ are irreducible modules of a completely new type.  
  
  We can slightly generalize these examples by applying the spectral-flow authomorphism $\rho_s$ of the Weyl algebras, which induces the spectral flow automorphism $\widetilde \rho_s$ of the Lie algebra $\widehat {\mathfrak gl}$.

  Let us summerize:
  \begin{theorem}  \label{main-intr}
  The vertex  algebra   $\mathcal W_{1+\infty,c=-1}$ has the following two families of irreducible modules of Whittaker types:
  \begin{itemize}
 \item[(1)] Tensor product modules $Z_1 \otimes Z_2$, where $Z_1$ is a Whittaker modules for the singlet algebra $\mathcal M(2)$ (cf. \cite{T2}) and $Z_2$ is a Whittaker module for the Heisenberg vertex algebra $M_1(1)$.
  \item[(2)] Modules $ \widetilde \rho_s (L(d, \bm{\lambda},\bm{\mu})) $  which are irreducible quotients of the  Weyl Whittaker modules $\rho_s(M_1(\bm{\lambda}, \bm{\mu}))$.
  \end{itemize}
 
  Moreover, the following statements are equivalent (cf. Propositions  \ref{identification-1},   \ref{non-isom}):
  \begin{itemize}
 \item[(a)]  Module in (2) is of type (1);
 \item[(b)] $M_1(\bm{\lambda}, \bm{\mu})$ has the structure of a $\Pi(0)$--module;
 \item[(c)] $\bm{\lambda} = (\lambda_0, 0, \cdots )$, $\lambda_0 \ne 0$.
 \end{itemize}
 \end{theorem}
 
 \subsection*{Typical $\mathcal W_{1+\infty, c}$--modules}
 Weak modules for the vertex operator algebra $\mathcal W_{1+\infty, c}$ can be naturally divided in two classes: typical and atypical. We say that a weak $\mathcal W_{1+\infty, c}$--module is typical if it is irreducible as a module for the Heisenberg-Virasoro vertex subalgebra of $\mathcal W_{1+\infty, c}$, which we denote by $\mathcal L^{HVir}_c$.
    In the case $c=-1$, one can construct typical/atypical highest weight $\mathcal W_{1+\infty, c}$--modules by using typical/atypical modules for the singlet vertex algebra (cf. \cite{A-singlet, AdM-JMP, Wa, CM04}).
      
 Tensor product modules in Theorem \ref{main-intr}(1) are typical since Whittaker $\mathcal M(2)$--modules are realized on irreducible, Whittaker Virasoro modules (cf. \cite{T2}). In  Theorem \ref{app} in  Appendix we prove typicality of our new series of modules:

\begin{theorem}
 Modules  $L(d, \bm{\lambda},\bm{\mu})$)  are typical $\mathcal W_{1+\infty, c=-1}$--modules.
 \end{theorem}
 
  
 \subsection*{A connection with Whittaker ${\mathfrak gl}(2 \ell)$--modules.} The study   of Whittaker modules $M_{1}(\bm{\lambda},\bm{\mu})$  as $\widehat{\mathfrak gl}$--modules is  motivated by \cite{ALPY} and the orbifold theory of   vertex operator algebras. As a result we get a nice family of Whittaker modules for the Lie algebra $\widehat{\mathfrak gl}$ and describe a complete set of Whittaker vectors. Since  one can embed finite-dimensional Lie algebra ${\mathfrak gl}(n)$ into $\widehat{\mathfrak gl}$, it is a natural question what is ${\mathfrak gl}(n)$--version of the Whittaker modules $M_{1}(\bm{\lambda},\bm{\mu})$. It turns out that the theoretical background for these modules was given  in \cite{BM} in the context of generalized Whittaker modules. But the paper \cite{BM} only presented a construction of universal generalized ${\mathfrak gl}(n)$--modules, and it seems the corresponding simple quotients haven't been investigated before our paper.
 Because of simplicity we consider here the ${\mathfrak gl}(2 \ell)$--modules $W(\bm{\alpha},\bm{\beta})$, where $\bm{\alpha},\bm{\beta} \in {\C} ^{\ell}$. which are quotients of universal generalized Whittaker modules from \cite{BM}. We present this construction in Section \ref{sect-gl2n}. We describe the structure of $W(\bm{\alpha},\bm{\beta})$ and obtained a complete family of Whittaker vectors.
We prove that if $\bm{\alpha},\bm{\beta} \ne 0$, then 
$$L(d,\bm{\alpha},\bm{\beta}) = \frac{   W(\bm{\alpha},\bm{\beta}) }  {\langle  (I-d) {\C}[I] {\whf} \rangle  }$$
is a simple generalized Whittaker ${\mathfrak gl}(2 \ell)$--module.

\vskip 5mm
{\bf Acknowledgements.}
We would like to thank  C. H. Lam, N. Yu,   A. Romanov and K. Zhao on discussion about Whittaker modules.

The authors  are  partially supported   by the
QuantiXLie Centre of Excellence, a project cofinanced
by the Croatian Government and European Union
through the European Regional Development Fund - the
Competitiveness and Cohesion Operational Programme
(KK.01.1.1.01.0004).

\section{Preliminaries}

 \subsection{Classical Whittaker modules }
 
 For the structural theory of Whittaker modules for Whittaker pairs see \cite{BM}.
 
Assume that $\g$ is a Lie algebra with triangular decomposition
$\g = \mathfrak n_- \oplus \mathfrak h \oplus \mathfrak n_+$.

Let $W$ be a $\g$--module. Vector $v \in W$ is called a \emph{Whittaker vector} if there is a Lie functional $\lambda : \mathfrak n_+ \rightarrow {\C}$ such that
$ x v = \lambda(x) v$ for all $x \in {\mathfrak n}_+$.

Let $\lambda : \mathfrak n_+ \rightarrow {\C}$ be a Lie functional. Let ${\C} v_{\lambda}$ be a $1$--dimensional  $\mathfrak n_+$--module such that
 $ x v_{\lambda} = \lambda(x) v_{\lambda}$ for all $x \in {\mathfrak n}_+$. Then
 $$ M_{\lambda} = U(\mathfrak g) \otimes _{ U(\mathfrak n_+)} {\C} v_{\lambda},$$
 is a $U(\g)$--module, which is called \emph{the universal  (or standard) Whittaker module}. 
 
 \subsection{Generalized Whittaker modules}
 
Let $\g$ be a Lie algebra, and $\mathfrak n$ any nilpotent subalgebra of $\g$. Whittaker modules with respect to  $\mathfrak n$ are sometimes called \emph{generalized Whittaker modules}, or Whittaker modules with respect to the pair $(\mathfrak g, \mathfrak n)$.  For example, we can take $\mathfrak n$ to be any commutative subalgebra of $ \mathfrak n_+$.

This makes this situation more general to the previous one where the nilpotent subalgebra is exactly $\mathfrak n_+$.

Let $W$ be a $\g$--module. Vector $v \in W$ is called \emph{Whittaker vector for the pair  $(\mathfrak g, \mathfrak n)$} if there is a Lie functional $\lambda : \mathfrak n\rightarrow {\C}$ such that
$ x v = \lambda(x) v$ for all $x \in {\mathfrak n}$.

Let $\lambda : \mathfrak n \rightarrow {\C}$ be a Lie functional. Let ${\C} v_{\lambda}$ be a $1$--dimensional  $\mathfrak n$--module such that
 $ x v_{\lambda} = \lambda(x) v_{\lambda}$ for all $x \in {\mathfrak n}$. Then
 $$ M_{\lambda} = U(\mathfrak g) \otimes _{ U(\mathfrak n)} {\C} v_{\lambda},$$
 is a $U(\g)$--module, which is called \emph{the universal  Whittaker module for the pair  $(\mathfrak g, \mathfrak n)$}.
 
 Such Whittaker modules and corresponding Whittaker vectors are sometimes called the generalized Whittaker modules/Whittaker vectors.

\subsection{Example: $\g = \mathfrak{gl}(2\ell , {\C})$}
\label{ex-gl2n}
In this Subsection we take $\g = \mathfrak{gl}(2\ell , {\C})$.
Take the usual basis $e_{i,j}$, $i,j = 1, \dots, n$ of  $\g$.


Recall the  triangular decomposition:
  $\g = \mathfrak n_- \oplus \mathfrak h \oplus \mathfrak n_+. $
  Here 
  \bea \mathfrak n_{+}  &=& \mbox{span}_{\C} \{ e_{i,j} \ \vert \ i <j \}, \nonumber \\
  \mathfrak n_{-}  &=& \mbox{span}_{\C} \{ e_{i,j} \ \vert \ i >j \}, \nonumber \\
  \mathfrak h   &=& \mbox{span}_{\C} \{ e_{i,i} \ \vert \ i=1, \dots, 2\ell \}. \nonumber \eea

Note that $\mathfrak n_+$ is not commutative. For example $[e_{1,2}, e_{2,3}] = e_{1,3}$.

Whittaker modules with respect to this triangular decomposition are called classical Whittaker modules.

But we can take the following subalgebra $\mathfrak n \subset \mathfrak n_+$.
$$ \mathfrak n = \mbox{span}_{\C} \{ e_{i, j + \ell} \ \vert  i, j = 1, \dots, \ell \}. $$
Then $\mathfrak n$ is a commutative Lie algebra. 

Whittaker modules with respect to the pair $(\g, \mathfrak n)$ are called the generalized Whittaker modules. Let us describe these modules. Take $\lambda \in \mathfrak n ^{*}$. Consider the universal Whittaker module $$ M_{\lambda} = U(\mathfrak g) \otimes _{ U(\mathfrak n)} {\C} v_{\lambda}.$$
Now we have two natural questions:
\begin{itemize}
\item Is  $ M_{\lambda}$ irreducible?
\item If $ M_{\lambda}$  is  reducible, describe   simple quotients of  $ M_{\lambda}$.
\end{itemize}
These modules appeared  in \cite[Example 23]{BM}. Quite surprisingly, the irreducibility analysis  has  not  been presented neither in \cite{BM} or at any other recent work. We shall see later  that $ M_{\lambda}$ is never irreducible, and describe its  simple quotients.

\subsection{Whittaker modules for $\g=\widehat{\mathfrak gl}$} 
\label{gl-pair}
The basis  is given by (cf. \cite{KR0, KR}): $$ C, \ E_{i,j}, \quad i,j \in {\Z}. $$

The  triangular decomposition:
  $\g = \mathfrak n_- \oplus \mathfrak h \oplus \mathfrak n_+. $
  Here 
  \bea \mathfrak n_{+}  &=& \mbox{span}_{\C} \{ E_{i,j} \ \vert \ i <j \}, \nonumber \\
  \mathfrak n_{-}  &=& \mbox{span}_{\C} \{ E_{i,j} \ \vert \ i >j \}, \nonumber \\
  \mathfrak h   &=& {\C} C \oplus \mbox{span}_{\C} \{ E_{i,i} \ \vert \ i \in {\Z} \}. \nonumber \eea

 Let $$ \mathfrak p = \mbox{span}_{\C}    \{ E_{-i-1, j} \ \vert i, j \in {\Z}_{\ge 0} \}. $$
 Then $\mathfrak p$ is a commutative subalgebra of $\mathfrak n_{+}$.
 
 As before we 
  call the  Whittaker modules for the Whittaker pair $(\g, \mathfrak n_+)$ the classical Whittaker $\widehat{\mathfrak gl}$--modules, and Whittaker modules for the pair $(\g, \mathfrak p)$ the generalized Whittaker $\widehat{\mathfrak gl}$--modules. 
 
 In what follows we shall describe a family of simple, generalized  Whittaker $\widehat{\mathfrak gl}$--modules.

   \section{Weyl vertex algebra and its Whittaker modules}
\label{Weyl}

We now recall the notion of Weyl vertex algebra as this will be our main object of study. To define this vertex algebra, first we remind the reader of the notion of Weyl algebra $\widehat{\mathcal{A}}$

Let $\mathcal L$ be the infinite-dimensional Lie algebra with generators
$$K, a(n), a^*(n),  \quad n \in \Z,$$
  such that $K$ is in the center where and  the only notrivial relations  are  
\[[a(n), a^{*} (m)] = \delta_{n+m,0} K , \quad\ n,m\in \mathbb{Z}. \]
The Weyl algebra $\widehat{\mathcal A}$ is:
$$\widehat{\mathcal A} = \frac{U(\mathcal L)}{ \langle K -1 \rangle}, $$
where $\langle K -1 \rangle$ is the two sided ideal generated by $K-1$. So,   in $\widehat{\mathcal{A}}$ we have $K=1$.

To construct the Weyl vertex algebra, we need a vector space, so we choose the simple Weyl algebra module $M$ generated by a cyclic vector $\bf 1$ such that 
\[ a(n) {\bf 1} = a  ^* (n+1) {\bf 1} = 0 \quad (n \ge 0), \]
that is, $M = {\C}[a(-n), a^*(-m) \ \vert \ n > 0, \ m \ge 0 ]$.

 Now, by the generating fields theorem, there is a unique vertex algebra $(M, Y, {\bf 1})$   where
the  vertex operator is given by $$ Y: M \rightarrow \mbox{End}(M) [[z, z ^{-1}]] $$
such that
$$ Y (a(-1) {\bf 1}, z) = a(z), \quad Y(a^* (0) {\bf 1}, z) = a ^* (z),$$
$$ a(z)   = \sum_{n \in {\Z} } a(n) z^{-n-1}, \ \ a^{*}(z) =  \sum_{n \in {\Z} } a^{*}(n)
z^{-n}. $$

We choose the following conformal vector of central charge $c=-1$ (cf. \cite{KR}): $$ \omega = \frac{1}{2} (a(-1) a^{*}(-1) - a(-2) a^{*} (0)) {\bf 1}.$$
Then $(M, Y, {\bf 1}, \omega)$  has the structure of a $\frac{1}{2} {\Z}_{\ge 0}$--graded vertex operator algebra. Weak and ordinary modules for  $(M, Y, {\bf 1}, \omega)$ can be defined as in the case of ${\Z}$--graded vertex operator algebras.

Following \cite{ALPY}, 
we define the Whittaker module for $\widehat{\mathcal{A}}$ to be the quotient  \[M_1(\bm{\lambda}, \bm{\mu}) = \sfrac{\widehat{\mathcal{A}}}{I},\]
 where $\bm{\lambda} = (\lambda_0, \dotsc, \lambda_n)$, $\bm{\mu} = (\mu_1, \dotsc, \mu_m)$ and $I$ is the left ideal \[I = \big\langle a(0)-\lambda_0, \dotsc, a(n) - \lambda_n, a^{*}(1)-\mu_1, \dotsc, a^{*}(m) - \mu_m, a(n+1), \dotsc, a^{*}(n+1), \dotsc \big\rangle.\]

 Let $\mathfrak n$ be the subalgebra of $\mathcal L$ generated by $a(n), a^*(n+1)$, $n \in {\Z}_{\ge 0}$.  
 Then $\mathfrak n$ is commutative, and therefore nilpotent subalgebra of $\mathcal L$.   
\begin{proposition} \cite{ALPY} We have:
\item[(1)]  $M_1(  \bm{ \lambda}, \bm{ \mu}) $ is a standard Whittaker module for the Whittaker pair $(\mathcal L, \mathfrak n)$  of level $K=1$ with the  Whittaker function $\Lambda = \bm{\lambda}, \bm{\mu} : \mathfrak n \rightarrow {\C}$:
\bea 
&&\Lambda(a(i)) = \lambda_i \quad (i=0, \dots, n), \ \Lambda(a(i)) = 0 \quad (i>n)\nonumber \\
&& \Lambda(a^*(i)) = \mu_i \quad (i=1, \dots, m), \ \Lambda(a^*(i)) = 0 \quad (i>m).
\nonumber \eea
\item[(2)] $M_1(  \bm{ \lambda}, \bm{ \mu}) $  is an irreducible $\widehat{\mathcal A}$--module.
\item[(3)] $M_1(  \bm{ \lambda}, \bm{ \mu}) $ is an irreducible weak module for the Weyl vertex operator algebra $M$.
\end{proposition}
  
  The Whittaker vector will be denoted by ${\bf w}_{ \bm{ \lambda}, \bm{ \mu} }$.

\da{  For each $s \in {\Z}$, the Weyl algebra $\widehat{\mathcal{A}}$ contains the following automorphism $\rho_s$:
  $$ \rho_s (a(n)) = a(n+s),  \quad \rho_s (a^*(n)) = a^*(n-s), \quad (n \in {\Z}).$$
  
  For any $\widehat{\mathcal{A}}$--module $N$, let $\rho_s(N)$ denote the $\widehat{\mathcal{A}}$--module $\rho_s(N)$ such that
  $ \rho_s(N) = N $ as a vector space and 
  $$ x . v = \rho_s(x) v, \quad x \in \mathcal A, \ v \in N. $$
  Clearly, $N$ is irreducible $\widehat{\mathcal{A}}$--module if and only if $\rho_s(N)$ is irreducible  $\widehat{\mathcal{A}}$--module.
  
  Let $\mathfrak n^{(s)}$ denote the commutative subalgebra of $\mathcal L$ generated by $a(n-s), a^*(n+1+s)$, $n \in {\Z}_{\ge 0}$.

  \begin{proposition}    
  For each $s \in {\Z}$ we have:
\item[(1)]  $\rho_s(M_1(  \bm{ \lambda}, \bm{ \mu}) )$ is a standard Whittaker module for the Whittaker pair $(\mathcal L, \mathfrak n^{(s)})$  of level $K=1$ with the  Whittaker function $\Lambda^{(s)} :  \mathfrak n \rightarrow {\C}$:
\bea 
&&\Lambda^{(s)} (a(i-s)) = \lambda_i \quad (i=0, \dots, n), \ \Lambda^{(s)} (a(i-s)) = 0 \quad (i>n)\nonumber \\
&& \Lambda^{(s)} (a^*(i+s)) = \mu_i \quad (i=1, \dots, m), \ \Lambda^{(s)} (a^*(i+s)) = 0 \quad (i>m).
\nonumber \eea
\item[(2)] $\rho_s(M_1(  \bm{ \lambda}, \bm{ \mu})) $  is an irreducible $\widehat{\mathcal A}$--module.
\item[(3)] $\rho_s(M_1(  \bm{ \lambda}, \bm{ \mu})) $ is an irreducible weak module for the Weyl vertex operator algebra $M$.
  
  \end{proposition}

  }
Let $\zeta_p = e^{2 \pi i / p}$ be $p$-th root of unity. Let $g_p$ be the automorphism of  the vertex operator  algebra $M$ which is uniquely determined by the following automorphism  of the Weyl algebra  $\widehat{\mathcal{A}}$:
$$ a(n) \mapsto \zeta_p a(n), \quad a^* (n) \mapsto \zeta^{-1} _p a^* (n) \quad (n \in {\Z}). $$
Then $g_p$ is the automorphism of $M$ of order $p$.

The following result was proved in \cite{ALPY} in the case $s=0$ and the proof for arbitrary $s \in {\Z}$ is completely analogous.

\begin{theorem} \label{weyl-irreducible}  \cite{ALPY} Assume that $\Lambda=(\bm{\lambda},\bm{\mu})\ne0$ and $s \in {\Z}$.
Then
 $\rho_s(M_{1}(\bm{\lambda},\bm{\mu}))$ is an irreducible weak module for
the orbifold subalgebra $M^{\Z_{p}}=M^{\left\langle g_{p}\right\rangle}$, for each $p\ge1$.
\end{theorem}

Let us define an operator $J^0:= :a a^*:$. Then the components of the field $$Y(J^0 , z) = \sum_{n \in {\Z}} J^0 (n) z^{-n-1}$$ satisfies  the commutation relations for the Heisenberg algebra at level $-1$. We have the following  subalgebra  of $M$:
$$M^0 = \mbox{Ker} _{M} J^0 (0). $$

Let $\zeta \in {\C}$, $\vert \zeta \vert = 1$,  which is {\bf not a root of unity}. Let $g$ be the automorphism of  the vertex operator  algebra $M$ uniquely determined by the following automorphism  of the Weyl algebra  $\widehat{\mathcal{A}}$:
$$ a(n) \mapsto \zeta a(n), \quad a^* (n) \mapsto \zeta^{-1}  a^* (n) \quad (n \in {\Z}). $$
Then $g$ is the automorphism of $M$  of infinite order, and the group $G= \langle g \rangle $ is isomorphic to $\Z$.
Clearly, we have
$$ M^0 \cong M^G. $$
    
Note also that
$$M^0 = \bigcap_{p=1}^{\infty}  M^{\Z_{p}},  \quad M \supset M^{\Z_{2}} \supset \cdots  M^{\Z_{p}}    \supset \cdots \supset M^0.  $$

Since  $M_{1}(\bm{\lambda},\bm{\mu})$  is  an irreducible $M^{\Z_{p}}$--module for each $p \ge 1$, it is natural to ask if $M_{1}(\bm{\lambda},\bm{\mu})$ is irreducible also as an $M^0$--module. However, in this paper we prove that $M_{1}(\bm{\lambda},\bm{\mu})$  is a reducible and indecomposable $M^0$--module.


 \da{ Let us finish this section by describing the connections between $\rho_s (M_1(  \bm{ \lambda}, \bm{ \mu}))$ and $M_1(  \bm{ \lambda}, \bm{ \mu})$ as modules for vertex subalgebras invariant for $J^0(0)$.
 As $M$--modules these modules  are constructed by applying the automorphism $\rho_s$ which has a nice  description in terms of generators of $\widehat{\mathcal A}$. But,  the automorphism $\rho_s$ does also have sense as automorphism of vertex $J^0(0)$--invariant subalgebras, which can be explain by using H. Li's $\Delta$--operator.
 
 First we notice that as $M$--modules  (see \cite{AdP-2019} for a proof):
 $$ (\rho_s (M_1(  \bm{ \lambda}, \bm{ \mu})), Y_s (\cdot,  z)) = ( M_1(  \bm{ \lambda}, \bm{ \mu}), Y( \Delta(-s J^0, z) \cdot, z))$$
 where
 $$ \Delta(h, z) = z^{h (0) } \exp \left(\sum_{n =1} ^{\infty} \frac{h(n)}{-n} (-z) ^{-n} \right). $$
 In particular, the action of $J^0(z)$ on $\rho_s (M_1(  \bm{ \lambda}, \bm{ \mu}))$ is given by
 $$J^0 _s(z) =  J^0(z) + \frac{(-z) ^{-1}}{-1}  Y( -sJ^0 (1) J^0, z) = J^0(z) + s z^{-1} \mbox{Id}. $$

 By restricting this realization on vertex subalgebras of $M$ which are $J^0(0)$--invariant we get:

\begin{proposition} \label{spectral-flow-rest}  Let $U$ be one of the subalgebras  $M^{\Z_{p}}$ or $M^0$. As an  $U$--module  $\rho_s(M_1(  \bm{ \lambda}, \bm{ \mu})) $ is obtained from $M_1(  \bm{ \lambda}, \bm{ \mu})$ as follows:
$$ (\rho_s (M_1(  \bm{ \lambda}, \bm{ \mu})), Y_s (v, z)) = ( M_1(  \bm{ \lambda}, \bm{ \mu}), Y( \Delta(-s J^0, z) v, z))$$
where $v \in  U$. 
\end{proposition}

\begin{remark}In the case $U=M^0$, the previous proposition shows that if we describe the structure of $M_1(  \bm{ \lambda}, \bm{ \mu})$ as a $U$--module (i.e., description of submodules, irreducible quotients) we will automatically describe the structure of $\rho_s(M_1(  \bm{ \lambda}, \bm{ \mu}))$ as $U$--module.
\end{remark}
}.

\section{ $\mathcal W_{1+\infty}$-algebra at central charge $c=-1$ and its Whittaker modules}

A peculiarity of the orbifold   $M^0$  is that it has two additional  important realizations which we plan to explore in our paper:
\begin{itemize}
\item $M^0$ is isomorphic to the vertex algebra  $\mathcal W_{1+\infty}$-algebra at central charge $c=-1$.
    
    \item $M^0$ is isomorphic to the simple module for the Lie algebra $\widehat{\frak gl}$, which is the central extension of the Lie algebra of infinite matrices.
\end{itemize}

\subsection{ The $\mathcal W_{1+\infty}$-algebra approach} 
The universal vertex algebra   $\mathcal W_{1+\infty}^ c$ is generated by the fields 
$$ J^k (z)  =   \sum_{n \in {\Z}} J^k (n) z^{-n-k-1} \quad (k \in {\Z}_{\ge 0}),$$
whose components satisfy the commutation relations for the Lie algebra $\widehat{\mathcal D}$ at central charge $c$,  which is a central extension of the Lie algebra of complex regular differential operators on $\C^*$  (cf. \cite{FKRW},  \cite{KR}). It has a simple quotient, which we denote by $\mathcal W_{1+ \infty, c}$.

\da{  

The components of the fields $J^0(z), J^1(z)$ satisfies the commutation relation for  the  Heisenberg Virasoro Lie algebra $\mathcal H$ generated by $J^0(r), J^1(r),C_1, C_2$ 
 with the commutation relations  $(r,s \in {\Z})$:
\bea  [J^0 (r), J^0(s)] &=& r \delta_{r+s,0} C_2, \nonumber \\ 
  \ [J^1(r), J^0(s)] &=& - s J^1(r+s),    \nonumber \\
  \  [J^1(r), J^1(s)] &=& (r-s) J^1 (r+s) + \frac{r^3-r}{12} \delta_{r+s, 0} C_1,\nonumber
\eea
 with central elements $C_1$ and $C_2$.
On each weak $\mathcal W_{1+ \infty, c}$--module the  actions of the central elements are given by
$ C_1= -2c, C_2 = c$.
Let $\mathcal L^{HVir} _c$ be the Heisenberg-Virasoro vertex subalgebra of $\mathcal W_{1+ \infty, c}$ generated by $J^0(z), J^1(z)$.

 \begin{definition}
  An irreducible  weak $\mathcal W_{1+ \infty, c}$--module $\mathcal M$  is called typical (resp. atypical) if it is irreducible (resp. reducible)  as a $\mathcal L^{HVir} _c$--module. 
 \end{definition}

}

It was proved by V. Kac and A. Radul in \cite{KR} that $M^0 \cong \mathcal W_{1+ \infty, c}$ for $c=-1$. As a consequence, we have that $M^0$   is generated    by the fields
$$ J^k (z) = Y(   a^*  (-k) a , z) = : \left( \partial_z ^k a^*(z) \right) a(z):=   \sum_{n \in {\Z}} J^k (n) z^{-n-k-1}  \quad (k \in {\Z}_{\ge 0}).$$

 \begin{remark} One can show that   modules obtained  in the decomposition of  $M$  as a $M^0$--module are atypical (cf. \cite{Wa}), but the modules constructed from the  irreducible relaxed modules in \cite{AdP-2019} are typical highest weight  $\mathcal W_{1+ \infty, c=-1}$--modules.
\end{remark}
 
Our  Whittaker modules for  the Weyl vertex algebra are automatically Whittaker  weak modules for  $\mathcal W_{1+ \infty, c=-1}$.

\subsection{ Approach using the Lie algebra  $\widehat{\frak gl}$}

Define the generating function
$$ E(z,w) = : a(z) a^*(w): = \sum _{i,j \in {\Z}} E_{i,j} z^{i-1} w^{-j}.$$ 
In other words, the operators $E_{i,j}$ are defined as
\bea E_{i,j}  = :a (-i) a^*(j):  \label{formula-Eij} \eea
These operators endow $M$ with the structure of a $\widehat{\frak gl}$--module at central charge $K=-1$ (see formula (2.7) in \cite{KR} for commutation relations for $\widehat{\frak gl}$).

 We have (cf. \cite{KR}):
$$ [E_{i,j},  a(-m)] = \delta_{j,m} a(-i), \quad  [E_{i,j}, a^* (m)]  = -  \delta_{i,m} a^* (j). $$

 \subsection{Connection between   $\widehat{\frak gl}$--modules and  $\mathcal W_{1+\infty}$--modules}
  \label{con-gl-winf}
In  \cite{FKRW},  \cite{KR} the authors demonstrated how one can construct $\mathcal W_{1+\infty}$--modules from $\widehat{\mathfrak gl}$--modules. Strictly speaking, their  approach was used only for quasi-finite modules, which are necessarily weight modules. Therefore, they formulated their results for a category of modules which does not include Whittaker modules. However, we extend their approach to a broader category which also includes Whittaker modules.

Following \cite{FKRW}, we have a homomorphism of Lie algebras
$ \Phi_0 :   \widehat{\mathcal D} \rightarrow \widehat{\frak gl}$, determined by
 \bea \Phi_0 (  t^m f(D) ) &=& \sum_{j \in {\Z} }  f(-j) E_{j-m, j}, \quad \Phi_0 (C) = K \label{formula-D},\eea
where $f \in {\C}[x]$, $D= t \partial_t$.   This homomorphism enables us to consider  $\widehat{\frak gl}$--modules as $\widehat{\mathcal D}$--modules, and therefore  as modules for the vertex algebra $\mathcal W_{1+ \infty} ^c$. 

It seems that the representation theory of $\widehat{\frak gl}$ is easier for analysis than the representation theory of $\mathcal W_{1+ \infty} ^c$. Nonetheless, a problem is in the fact that $\Phi_0$ is not surjective homomorphism. Hence the restriction of irreducible  $\widehat{\frak gl}$--module to  $\widehat{\mathcal D}$ (and therefore to $\mathcal W_{1+ \infty} ^c$) need  not be irreducible. However, for a family of quasi-finite modules, V. Kac and A. Radul \cite{KR0} showed that the image of $\Phi_0$ is dense in a certain way which we explain latter (cf. \cite[Proposition 4.3]{KR0}). So for quasi-finite modules,  irreducible $\widehat{\frak gl}$--modules are also irreducible $\mathcal W_{1+ \infty} ^c$--modules. Our goal is to show that we can replace quasi-finite modules with certain Whittaker modules, so that irreducible $\widehat{\mathfrak gl}$--modules will become irreducible for  $\mathcal W_{1+ \infty} ^c$.

 \section{  $M_{1}(\bm{\lambda},\bm{\mu})$  as a Whittaker $\widehat{\frak gl}$--module }

  \begin{proposition}  \label{ired-gl}
  Assume that $\bm{\lambda} \ne 0$  and  $\bm{\mu} \ne0$.
Then $M_{1}(\bm{\lambda},\bm{\mu})$ is a Whittaker $\widehat{\frak gl}$--module for the pair $\mathfrak p \subset \widehat{\frak gl}$  at central charge $K=-1$, generated by the Whittaker vector ${\wh}$.
\end{proposition}
\begin{proof}
 
Let us prove that $\wh$ is indeed cyclic. As a vector space, $ M_{1}(\bm{\lambda},\bm{\mu}) \cong M$,
 so its basis elements are of the form

\[ w = a(-n_1)  \cdots a(-n_r) a^* (-m_1) \cdots a^* (-m_s) {\wh} .\]

We will now use double induction on the length of a vector in $ M_{1}(\bm{\lambda},\bm{\mu})$. First inductive hypothesis is the following:

Every element of the form  $ w = a(-n_1)  \cdots a(-n_r){\wh}$ is an element of $ U(\widehat{\gl}). {\wh}$.

Let us first prove the initial case. Let $a(-n)\wh$ be an element of $M_{1}(\bm{\lambda},\bm{\mu})$. Since $\bm{\mu} \neq 0$, there is an index $j_0$ such that $\mu_{j0} \neq 0$, and $a^{*}(j_0)\wh = \mu_j \wh$. Therefore, we have
\[\frac{1}{\mu_{j_0}}E_{i, j0}\wh = a(-n)\wh.\]

Now, let us assume that for any vector of lenght at most $r$, the inductive hypothesis holds. Let us take a vector of the form 
\[v = a(-n_1)a(-n_2)\dotsb a(-n_{r+1})\wh.\]
We have that
\begin{align*}
& \frac{1}{\mu_{j_0}^{r+1}}E_{n_1, j_0}E_{n_2, j_0}\dotsb E_{n_{r+1}, j_0}\wh \\
& = \frac{1}{\mu_{j_0}^{r+1}}:a(-n_1)a^*(j_0):a(-n_2)a^*(j_0):\dotsc:a(-n_{r+1})a^*(j_0):\wh \\
& = \frac{1}{\mu_{j_0}^{r+1}}a(-n_1)a(-n_2)\dotsb a(-n_{r+1}) \underbrace{a^*(j_0)a^*(j_0)\dotsb a^*(j_0)}_{r+1 \text{factors}}\wh \\
& + C_1\frac{1}{\mu_{j_0}^{r+1}}a(-n_1)\dotsb \widehat{a(-n_k)}\dotsb a(-n_{r+1})\underbrace{a^*(j_0)\dotsb a^*(j_0)}_{r \text{factors}}\wh \\
& + \dotsc + C_{r+1}\wh, \\ 
& = \frac{1}{\mu_{j_0}^{r+1}} \mu_{j_0}^{r+1}v + K,
\end{align*}
where $C_i$ are constants that can be zero, $\widehat{a(-n)}$ means that this element is left out, and $K$ is a (possibly zero) sum of vectors of length at most $r$.
Therefore, we have that 
\[v =  \frac{1}{\mu_{j_0}^{r+1}}E_{n_1, j_0}E_{n_2, j_0}\dotsb E_{n_{r+1}, j_0}\wh  - K,\]
and by the induction hypothesis, we have that $v \in U(\widehat{\gl}). {\wh}$.

The second induction hypothesis is that every vector of the form 
\[a^*(-m_1)\dotsb a^*(-m_s)a(-n_1)\dotsb a(-n_r)\wh\]
is in $U(\widehat{\gl}). {\wh}$.

Let us first prove the initial case. Let $v = a^*(-m)a(-n_1)a(-n_2)\dotsb a(-n_r)\wh$, where $r$ is any non-negative integer. Since $\bm{\lambda} \ne 0$, there is an index $\lambda_{i0} \ne 0$. Therefore,
\begin{align*}
& \frac{1}{\lambda_{i_0}}E_{-i_0, -m}a(n_1)\dotsb a(-n_r)\wh \\
& = \frac{1}{\lambda_{i_0}}:a(i_0)a^*(-m):a(n_1)\dotsb a(-n_r)\wh \\
& = v.
\end{align*}

Now, let us assume that for any vector with at most $s$ factors $a^*()$ the inductive hypothesis holds. Let us take a vector of the form 
\[v = a^*(-m_1)\dotsb a^*(-m_{s+1})a(-n_1)a(-n_2)\dotsb a(-n_r)\wh.\]
We have that
\begin{align*}
& \frac{1}{\lambda_{i_0}^{s+1}}E_{-i_0, -m_1}E_{-i_0, -m_2}\dotsb E_{-i_0, -m_{s+1}}a(-n_1)a(-n_2)\dotsb a(-n_r)\wh \\
& = \frac{1}{\lambda_{i_0}^{s+1}}:a(i_0)a^*(-m_1):a(i_0)a^*(-m_2):\dotsc:a(i_0)a^*(-m_{s+1}):a(-n_1)\dotsb a(-n_r)\wh  \\
& = \frac{1}{\lambda_{i_0}^{s+1}}a^*(-m_1)a^*(-m_2)\dotsb a^*(-m_{s+1}) a(-n_1)\dotsb a(-n_r) \underbrace{a(i_0)a(i_0)\dotsb a(i_0)}_{s+1 \text{factors}}\wh \\
& + D_1\frac{1}{\lambda_{i_0}^{s+1}}a^*(-m_1)\dotsb \widehat{a^*(-m_k)}\dotsb a^*(-m_{s+1})a(-n_1)\dotsb a(-n_r)\underbrace{a(i_0)\dotsb a(i_0)}_{s \text{factors}}\wh \\
& + \dotsc + D_{s+1}\wh, \\ 
& = \frac{1}{\lambda_{i_0}^{s+1}}\lambda_{i_0}^{s+1}v + L,
\end{align*}
where $D_i$ are constants that can be zero, $\widehat{a^*(-m)}$ means that this element is left out, and $L$ is a (possibly zero) sum of vectors of length at most $s$.
Therefore, we have that 
\[v =  \frac{1}{\lambda_{i_0}^{s+1}}E_{-i_0, -m_1}E_{-i_0, -m_2}\dotsb E_{-i_0, -m_{s+1}}a(-n_1)a(-n_2)\dotsb a(-n_r)\wh  - L,\]
and by the induction hypothesis, we have that $v \in U(\widehat{\gl}). {\wh}$.

This completes the proof that the Whittaker vector $\wh$ is cyclic.
\end{proof}

Next, using a connection between $\widehat{\frak gl}$--modules and $\mathcal W_{1+ \infty}$--modules we have:

\begin{theorem} \label{structure-equivalence} Assume that $\bm{\lambda} \ne 0$  and  $\bm{\mu} \ne0$.
Then $M_{1}(\bm{\lambda},\bm{\mu})$ is a Whittaker  $ \mathcal W_{1+ \infty, c=-1}$--module generated by the cyclic vector ${\wh}$.  In particular, $M_{1}(\bm{\lambda},\bm{\mu})$  is a cyclic  module for the orbifold vertex algebra $M^0$.
\end{theorem}
\begin{proof}
By Proposition \ref{ired-gl} we  have  that $M_{1}(\bm{\lambda},\bm{\mu})$ is a cyclic  $\widehat{\frak gl}$--module and we want to prove it is a cyclic $\widehat{\mathcal D}$--module.
Take a homomorphism $ \Phi_0 :   \widehat{\mathcal D} \rightarrow \widehat{\frak gl}$.  Although $M_{1}(\bm{\lambda},\bm{\mu})$ is not a quasi-finite weight module, we  will show that one can still apply \cite[Proposition 4.3]{KR0}. We use the following arguments:

\begin{itemize}

\item Let $\mathcal M = M_{1}(\bm{\lambda},\bm{\mu})$.
 Recall that  $\mathcal M \cong M$ as a vector spaces.
 
 \item Take the Virasoro vector  $\omega = \frac{1}{2} ( a(-1) a ^*(-1) - a(-2) a^*(0)) {\bf 1}$ of central charge $c=-1$ and define $L(n) = \omega_{n+1}$. Then $L(0)$ defines a $\tfrac{1}{2} {\Z}_{\ge 0}$ gradation on $M$:
 $$ M = \bigoplus_{m \in \tfrac{1}{2} \Z_{\ge 0}}  M_ m$$
 such that $v \in M_m \iff L(0)v = m v$.

   \item Define $\mathcal M (m) = \bigcup_{ k \le m} M_k$.
   
    \item Since $\dim M_k  < \infty$ for all  $k \in \tfrac{1}{2} \Z_{\ge 0}$, we have that  $ \dim \mathcal M (m)    < \infty$ for all  $m \in \tfrac{1}{2} \Z_{\ge 0}$.

 \item  Then  by applying the homomorphism  $\Phi_0$ on $M_{1}(\bm{\lambda},\bm{\mu})$ we get:
  \bea \Phi_0 (  t^m f(D) ) &=& \sum_{j \in {\Z} }  f(-j) E_{j-m, j}, \quad \Phi_0 (C) = -1 \label{formula-D-primjena} \\ 
  &=& \sum_{j \in {\Z} }  f(-j) :a(m-j) a^* (j) : \nonumber 
  \eea

\item For $v \in \mathcal M(m)$, we have that $\Phi_0 (  t^m f(D) )  v \in \mathcal M(m + N)$, where $N$ is determined by Whittaker function, has the property
$$ \lambda_i = \mu_i = 0 \quad \forall i > N,$$
\item[]  and can be chosen independently on $m$.

\item Therefore  $\Phi_0 (  t^m f(D) )$ defines a homomorphism $$ (*) \quad {\C}[x] \rightarrow \mbox{Hom} (\mathcal M(m), \mathcal M(m+N)).$$
\item Since  $\dim \mathcal M(m) < \infty$ and $\dim  \mathcal M(m+N) < \infty$, the proof of \cite[Proposition 4.3]{KR0} gives that
(*) is a continuous map. Therefore, it can be uniquely  extended to a homomorphism  
$$ (*) \quad \mathcal O \rightarrow \mbox{Hom} (\mathcal M(m), \mathcal M(m+N)), $$
where $\mathcal O$ denotes the algebra of  holomorphic functions on ${\C}$.

\item Since for each $j \in {\Z}$ there is a holomorphic function $g$ such that $g(\ell ) = \delta_{\ell,j}$, $\ell \in {\Z}$, we conclude that each $E_{i,j}$ is in the image of  map (\ref{formula-D-primjena}).

\item Since $\mathcal M$ is a Whittaker $\widehat{\frak gl}$--module generated by  ${\wh}$,   $\mathcal M$ is also generated by  ${\wh}$   as a ${\widehat{\mathcal D}}$--module.

\end{itemize}
The proof follows.
\end{proof}

\section{The structure of the Whittaker module  $M_{1}(\bm{\lambda},\bm{\mu})$ as  a $\widehat{\frak gl}$--module }
 
We have proved  in Proposition \ref{ired-gl} that $ M_{1}(\bm{\lambda},\bm{\mu})$ is a cyclic  $\widehat{\frak gl}$--module. Let us prove that it is not irreducible.
We will use the following Casimir element.

Define
$$ I := \sum_{j \in {\Z} } E_{j,j}. $$

  As far we can see, $I$ is introduced in \cite{GS} for a slightly different category of modules. However, it is well defined on the family of Whittaker $\widehat{\frak gl}$--modules which we considered. 

\begin{lemma}
 \da{On any weak $M$--module $\mathcal M$ we have that
 $ I = J^0(0)$. In particular:}
\begin{itemize}
\item[(1)]  $I \in \mbox{End} (M_{1}(\bm{\lambda},\bm{\mu}))$.
\item[(2)]  The action of $I$ commutes with the action of $\widehat{\frak gl}$ on $M_{1}(\bm{\lambda},\bm{\mu})$.
\end{itemize}
\end{lemma}
 \begin{proof}
 
 On each weak $M$--module the operator  $ J^0(n) = \sum_{j \in {\Z}} : a(j+n) a(-j):$ is well defined. In particular, we have:
 $$ J^0(0) = \sum_{j \in {\Z}} : a(-j) a^* (j): = I. $$
 In particular this implies that $I$ is well defined operator acting on   $ M_{1}(\bm{\lambda},\bm{\mu})$. Thus (1) holds.
 
 For the proof of assertion (2) we use commutation relations for  $\widehat{\frak gl}$:
 $$ [ E_{i,j}, E_{s,t} ]  =  \delta_{j,s}  E_{i,t} - \delta_{i,t} E_{s,j} - C  \Phi( E_{i,j}, E_{s,t}). $$
We calculate:
 \bea
  [ E_{i,j}, I]  v&=& \left( \sum_{k \in {\Z} } [ E_{i,j}, E_{k,k} ] \right)  v  \nonumber \\
  &=&  \left(   E_{i, j } -  E_{i , j} \right) v \nonumber  \\
  &=&  0. \nonumber  
 \eea
 
 Therefore, our second assertion also holds.
 
 \end{proof}
 
\begin{lemma} \begin{itemize}
\item[(1)] For every $n \in {\Z}_{\ge 1}$,
$I^n {\wh}$ is an non-trivial Whittaker vector in $M_{1}(\bm{\lambda},\bm{\mu})$ . 
\item[(2)]  For any 
$S \subset   {\C} [I] {\wh} $, let $\langle S \rangle$  be the submodule generated by  Whittaker vectors from $S$.  Then $\langle (I-d)  {\C} [I] {\wh} \rangle $  is a proper submodule of $M_{1}(\bm{\lambda},\bm{\mu})$ for each $d \in {\C}$ .
\end{itemize}
\end{lemma}
We have proved the following result.

\begin{theorem} 
$ M_{1}(\bm{\lambda},\bm{\mu})$  is a reducible, cyclic  $\widehat{\frak gl}$--module.
\end{theorem}

 \da{ The automorphism $\rho_s$ of $\widehat{\mathcal A}$ induces the following automorphism of $\widehat{\frak{gl}}$:
 $$\widetilde  \rho_s :  E_{i,j} = : a(-i) a^*(j) \mapsto E_{i-s, j-s} = :a (-i+s) a^{*} (j-s):. $$
 So this means that as $\widehat{\frak{gl}}$--module $\rho_s (M_{1}(\bm{\lambda},\bm{\mu}))$ is obtained from   $M_{1}(\bm{\lambda},\bm{\mu}) $ by applying the automorphism $\widetilde  \rho_s$.
As a consequence we gat
 \begin{corollary} \label{primjena-gl-aut}We have:
 \item[(1)] $\rho_s (M_{1}(\bm{\lambda},\bm{\mu}))$ is a reducible, cyclic  $\widehat{\frak gl}$--module.
 \item[(2)]  Let  $\mathcal L$ be any  irreducible quotient of $M_{1}(\bm{\lambda},\bm{\mu})$. Then $\widetilde  \rho_s (\mathcal L)$ is a irreducible quotient of  $\rho_s (M_{1}(\bm{\lambda},\bm{\mu}))$.
 \end{corollary}
 
 }

 Let  $\mathcal L$ be any  irreducible quotient of $M_{1}(\bm{\lambda},\bm{\mu})$. (We will see bellow  that it is not unique). Two  very important problems arise:
 
  \begin{itemize}
  \item[(A)] Find an explicit realization of $\mathcal L$   if possible.
  \item[(B)]  Determine the  complete set of Whittaker vectors in $M_{1}(\bm{\lambda},\bm{\mu})$ which  generate the maximal submodule of $M_{1}(\bm{\lambda},\bm{\mu})$. 
 \end{itemize}
   \da{In what follows we will solve both problems.}

 \section{Whittaker vectors in $M_{1}(\bm{\lambda},\bm{\mu})$ }

In this section we describe the complete set of Whittaker vectors in $M_{1}(\bm{\lambda},\bm{\mu})$ and determine its simple quotients. In particular, we completely solve the problem (B).

Let $\mathcal M = M_{1}(\bm{\lambda},\bm{\mu})$.

As before, we choose $N \in {\Z}_{\ge 0}$ such that $a(n) {\wh} = a^*(n) {\wh} = 0$ for $n \ge N$.
Then \bea  v =  I   {\wh} &=& \left( : a(0) a^*(0) :  + \cdots + :a(N) a^*(-N): \right){\wh} \nonumber \\ && +  \left( : a(-1) a^*(1) :    + \cdots + :a(-N) a^*(N): \right){\wh} \nonumber \\ &=& \sum_{k=0} ^N \lambda_k a^*(-k) {\wh} +  \sum_{k=1} ^N \mu_k a(-k) {\wh} \eea
is a new Whittaker vector. Now we choose $i_0$ and $j_0$ such that $$\lambda_{i_0} \cdot \mu_{j_0} \ne 0. $$

\begin{lemma}
For $n, k \ge 0$ we have
$$ a(n)  I ^k {\wh} = {\lambda_n} (I+1) ^k  {\wh}. $$
$$ a^*(n+1)  I ^k {\wh} = {\mu_{n+1}} (I-1) ^k  {\wh}. $$
\end{lemma}
\begin{proof} By direct calculation we have
$$ [a(n), I ] = \sum _{j \in {\Z} } [a(n), E_{j,j}] = [a(n), a^*(-n) a(n)] = a(n).$$
This implies that
$$ a(n) I {\wh} = \lambda_n  I {\wh} + [a(n), I] {\wh} = \lambda_n (I {\wh} + {\wh}) =
\lambda_n ( I +1) {\wh}. $$
$$ a(n) I^2 {\wh}   = {\lambda_n} (I ^2    + 2 I + 1) {\wh} = {\lambda_n} (I+1) ^2 {\wh}
$$
Assume that  $a(n)  I ^k {\wh} \in   {\bf {\C}[I] {\wh}}$. Then
\bea  a(n) I^{k+1} {\wh}  &=& I (a(n) I^k {\wh}) + a(n) I^k {\wh} = {\lambda_n} (I (I+1) ^k + (I+1)^k)  {\wh}. \nonumber \\
&=& (I+1) ^{k+1} {\wh}. \nonumber \eea
Now first  claim holds by the induction. The proof of the second claim is analogous.
\end{proof}

  Consider now some examples
\bea  E_{-i_0, n_1} a(-n_1) I^k {\wh}  &=&  [ E_{-i_0, n_1}, a(-n_1)] ^k  I^k {\wh}      + a(-n_1)  E_{-i_0, n_1}   I^k {\wh} \nonumber \\
&=& - a(i_0)   I^k {\wh}   +  a(-n_1)     I^k E_{-i_0, n_1}  {\wh}  \nonumber \\
&=& -\lambda_{i_0} (I+1) ^k {\wh} + \lambda_{i_0}  \mu_{n_1} a(-n_1) I^k {\wh}  \nonumber \eea
  
 
\bea  E_{m_1, j_0} a^*(-m_1) I^k {\wh}  &=&  [ E_{m_1, j_0}, a^*(-m_1)]   I^k {\wh}      + a^*(-m_1)  E_{m_1, j_0}   I^k {\wh} \nonumber \\
&=& a^* (j_0)   I^k {\wh}   +  a^*(-m_1)     I^k E_{m_1, j_0}  {\wh}  \nonumber \\
&=& \mu_{j_0} (I-1) ^k {\wh} + \lambda_{m_1} \mu_{j_0}   a^*(-m_1) I^k {\wh}  \nonumber \eea
This implies:
$$  \widehat  E_{m_1, j_0}   a^*(-m_1) I^k {\wh}  =  \frac{1}{\mu_{j_0}} 
  ( E_{m_1, j_0 } - \lambda_{m_1} \mu_{j_0} )  a^* (-n_1) I^k {\wh} = ( I-1) ^k {\wh}, $$
  and 
   
In this way we get:

\begin{lemma} \label{derivacije} Let   $p \in {\Z}_{\ge 0}$.
\item[(1)] Let $$\Phi \in {\C}[a(-n-1),  a^*(-m) \ \vert n,m  \in {\Z}_{\ge 0}, m \ne i_0 \}. $$
Then 
  $$\widehat E_{-i_0, p}\  \Phi   I ^k  {\wh} = \left( \frac{\partial }{\partial a(-p) } \Phi\right)
  \ (I+1)^{k} {\wh}. $$
 \item[(2)]  Let $$\Phi^* \in {\C}[ a^*(-m) \ \vert m  \in {\Z}_{\ge 0}, m \ne i_0 \}. $$
Then 
  $$\widehat E_{p, j_0} \  \Phi^*   I ^k  {\wh} = \left( \frac{\partial }{\partial a^*(-p) } \Phi\right)
  \ (I-1)^k {\wh}. $$
\end{lemma}

\begin{proposition}  \label{whit-opis} Assume that $v$ is a Whittaker vector in $\mathcal M$. Then
$ v \in   {\bf {\C}[I] {\wh}}$. 
\end{proposition}

\begin{proof} 

First we see that as a vector space:
$$\mathcal M \cong  {\C}[a(-n-1),  a^*(-m) \ \vert n,m  \in {\Z}_{\ge 0}, m \ne i_0 \} \otimes {\C}[I], $$
i.e., the basis of  $\mathcal M$ consists of the vectors:
$$(*) \ a(-n_1) \cdots a(-n_r) a^*(-m_1) \cdots a^*(-m_s) I ^k  {\wh}, \ n_i \ge 1, m_i \ge 0, m_i \ne i_0,  k \ge 0. $$

 So we can assume that we have a Whittaker vector of the form
\bea v= \sum_{j=0} ^{k} \Phi_j   I^j {\wh} \label{form-whitt} \eea for certain polynomials  $\Phi_j     \in {\C}[a(-n-1),  a^*(-m) \ \vert n,m  \in {\Z}_{\ge 0}, m \ne i_0 \}$. By using Lemma \ref{derivacije} (1) we get
  $$\widehat E_{-i_0, p}\  v =  \sum_{j=0} ^{k} \left( \frac{\partial }{\partial a(-p) } \Phi_j\right)
  \ (I+1) ^j {\wh}. $$
  Since $v$ is a Whittaker vector  $\widehat E_{-i_0, p}$ must act trivially on it. Since vectors (*) are linearly independent we conclude that   
 $$ \frac{\partial }{\partial a(-p) } \Phi_j = 0  \quad \forall p \in {\Z}_{>1}, \ j=0, \dots,k. $$
 This implies that $\Phi_j = \Phi_j ^*$  for certain $\Phi_j ^* \in {\C}[ a^*(m) \ \vert m  \in {\Z}_{\ge 0}, m \ne i_0 \} $ and  
  $$v= \sum_{j=0} ^{k} \Phi^* _j   I ^j {\wh}   $$
   
  Now using Lemma \ref{derivacije} (2)  we get 
  $$   \widehat E_{p, j_0} v =  \sum_{j=0} ^{k} \left( \frac{\partial }{\partial a^* (-p) } \Phi_j^* \right)
  \ (I-1) ^j   {\wh}. $$
 Since $v$ is a Whittaker vector, we have  $\widehat E_{p, j_0} v=0$, implying that
$ \frac{\partial }{\partial a^* (-p) } \Phi_j^* = 0 \quad \forall p \ge 0$. This implies that $\Phi_j ^*$ is a constant, and therefore $ v \in   {\bf {\C}[I] {\wh}}$. The Claim holds.
\end{proof} 

As a consequence, we get description of irreducible Whittaker modules.
\begin{theorem} \label{irred-description}
For each $d \in {\C}$:
$$L(d,\bm{\lambda},\bm{\mu}) = \frac{   M_{1}(\bm{\lambda},\bm{\mu}) }  {\langle  (I-d) {\C}[I] {\wh} \rangle  }$$
is irreducible.
\end{theorem}
\begin{proof}

It suffices  to prove that each vector $v$ of the form  (\ref{form-whitt}), whose projection to  $L(d,\bm{\lambda},\bm{\mu}) $ is non-zero,  is cyclic. By using exactly the same arguments as in the proof of Proposition \ref{whit-opis} we see that $\wh \in U(\widehat{\mathfrak{gl}}). v $. Since $\wh$ is cyclic vector in  $M_{1}(\bm{\lambda},\bm{\mu})$, we conclude that $v$ is cyclic. The proof follows.
\end{proof}

\da{
By Applying Corollary  \ref{primjena-gl-aut} on the irreducible $\widehat{\frak{gl}}$--module $L(d,\bm{\lambda},\bm{\mu})$ we get:
\begin{corollary}
For any $s \in {\Z}$, $\widetilde \rho_s( L(d,\bm{\lambda},\bm{\mu}) )$ is an irreducible quotient of  $\rho_s( M_{1}(\bm{\lambda},\bm{\mu}))$.
\end{corollary}
}

 \section{Bosonic realisations of  $M_{1}(\bm{\lambda},\bm{\mu})$ and $L(d,\bm{\lambda},\bm{\mu})$ }
 
 \label{sect-bosonic}
K. Tanabe in \cite{T2} showed that the Whittaker modules for the  rank one Heisenberg vertex algebra remain irreducible when we restrict them to the singlet vertex algebra  introduced in  \cite{A-singlet}. Since $\mathcal W_{1+\infty}$-algebra at central charge $c=-1$ is isomorphic to the tensor product  $\mathcal W_{3,-2} \otimes M_1(1)$, where $\mathcal W_{3,-2}$ is the singlet vertex algebra  at central charge $ -2$ and $M_1(1)$ is rank one Heisenberg vertex algebra (cf. \cite{Wa}), $\mathcal W_{1+\infty, c=-1}$ can be realized as a  subalgebra of the rank two Heisenberg vertex algebra $M_2(1)$.

 Using Tanabe's result, one shows that irreducible Whittaker $M_2(1)$--modules remain irreducible when we restrict them to the vertex algebra $\mathcal W_{1+\infty, c=-1}$. But it is a hard problem to identify them with our Whittaker modules $L(d,\bm{\lambda},\bm{\mu})$  realised as quotients of Whittaker modules for the Weyl vertex algebra.  In this article, we identify modules $L(d,\bm{\lambda},\bm{\mu})$ as $M_2(1)$--modules in the cases when $M_{1}(\bm{\lambda},\bm{\mu})$ has the structure of a $\Pi(0)$--module which will be described  Subsection \ref{modules-pi-whitt}. But in general, our Whittaker modules are different to those obtained from Tanabe paper.

\subsection{Realisation of $M_{1}(\bm{\lambda},\bm{\mu})$ from $\Pi(0)$--modules}
\label{modules-pi-whitt}
Let $V_L= M_{\gamma,\delta} (1) \otimes {\C}[L]$ be the lattice vertex algebra associated to  the lattice $L={\Z}\gamma  + {\Z} \delta$ with products $$\langle \gamma, \gamma \rangle = \langle \delta, \delta \rangle =0, \quad \langle \gamma, \delta \rangle = 2,$$
where $ M_{\gamma,\delta} (1)$ is the Heisenberg vertex algebra generated by $\gamma(z), \delta(z)$, and ${\C}[L]$ the group algebra of the lattice $L$.

Consider the following subalgebra of $V_L$:
 $$\Pi(0) = M_{\gamma,\delta} (1) \otimes {\C}[{\Z} \gamma]. $$
Let $M_{\varphi} (0)$ be  the commutative vertex algebra   generated by the commutative even field $\varphi(z) =\sum_{i \in {\Z}} \varphi(i) z^{-i-1}$, so
$$ M_{\varphi} (0) = {\C}[\varphi(-1), \cdots]. $$
For ${\chi} =\chi(z) = \sum_{n \in {\Z}} \chi_i z^{-i-1} \in {\C}((z))$, let  $L_{\varphi} (\chi)= {\C}  {\bf 1}_{\chi} $ be the $1$--dimensional  $M_{\varphi} (0)$--module such that $\varphi(i) \equiv \chi_i  \mbox{Id}  $ on $L_{\varphi} (\chi)$.

Note that  $\gamma, \overline \delta = \delta - 2 \varphi, e^{\pm \gamma}$ generate a subalgebra $\widetilde \Pi(0)$ of  $M_{\varphi} (0) \otimes \Pi(0)$ and that $\widetilde \Pi(0) \cong \Pi(0)$.

There is an embedding of vertex algebras $\Phi: M \hookrightarrow \widetilde \Pi(0) \hookrightarrow M_{\varphi} (0) \otimes \Pi(0)$ given by
 $$ \Phi(a) = e^{\gamma}, \Phi(a^*) = -\frac{1}{2} (\gamma (-1) + \delta(-1) -2 \varphi(-1)) e^{-\gamma}.$$

 Let $a^{-1} = e^{-\gamma}$. Define $a (n) = e^{\gamma}_n, a^{-1} (n) = e^{-\gamma}_{n-2}$. 
For ${\lambda} \in {\C}$, $\lambda \ne 0$,  let 
$\Pi_{\lambda}$
be the ${\Z}_{\ge 0}$--graded irreducible Whittaker  $\Pi(0)$--module (cf. \cite{ALZ}), generated by the Whittaker vector $v_{\lambda}$ such that
 $$
   a(n) v_{\lambda}  =   \lambda \delta_{n,0}  v_{\lambda}, \quad a^{-1} (n)    v_{\lambda} = \frac{1}{\lambda} \delta_{n,0} v_{\lambda} \quad (n \in {\Z}_{\ge 0}).$$
 As   vector spaces:
 $$ \Pi_{\lambda} \cong   {\C}[ \gamma (-n), \delta(-n+1) \ \vert \ n \in {\Z}_{>0}],    \quad
 (\Pi_{\lambda} )_{top}  = {\C}[\delta(0)] v_{\lambda} \cong {\C}[\delta(0)]. $$

Now consider module $L_{\varphi} (\chi) \otimes \Pi_{\lambda}$. Define $w_{\lambda, \chi} =  {\bf 1}_{\chi}  \otimes v_{\lambda}$.

Assume that  $\bm{\mu} =(\mu_1, \mu_2, \cdots, \mu_n)$ 
$$ \chi_i = \lambda \mu_i, \ i=1, \dots, n, \chi_j = 0 \quad \forall j > n. $$

We have:
$$
a^*(i)  w_{\lambda, \chi} = \mu_i  w_{\lambda, \chi}, \quad i=1, \dots, n, \quad a^*(i)  w_{\lambda, \chi} = 0, \quad \mbox{for} \ i > n.
$$

\begin{proposition}   Assume that $\lambda \ne 0$, $\mu_n \ne 0$, $\bm{\lambda} =(\lambda, 0, 0, \dots )$, $\bm{\mu} =(\mu_1, \mu_2, \cdots, \mu_n)$ and 
$$\chi(z) = \sum_{i\in {\Z}_{\le 0}}  \chi_i z^{-i-1}  + \frac{1}{\lambda} \sum_{i=1} ^{n} \mu_i z^{-i-1},
 $$
 Then  we have:
 \begin{itemize}
 \item   $M_{1}(\bm{\lambda},\bm{\mu}) \cong L_{\varphi} (\chi) \otimes \Pi_{\lambda}$
 as  modules for the Weyl vertex algebra $M$.
 \item $M_{1}(\bm{\lambda},\bm{\mu})$  has the structure of an irreducible $\Pi(0)$--module.
 \end{itemize}
\end{proposition}

\begin{remark} \label{arguments-1}
Note that if $\bm{\lambda}$ has  not of the form  $(\lambda, 0, 0, \dots )$, then the $M$--module structure on  $M_{1}(\bm{\lambda},\bm{\mu})$ can not be extended to a structure of  $\Pi(0)$--module.  We simply can not define the vertex operator
\bea Y_{M_{1}(\bm{\lambda},\bm{\mu})} (e^c, z) \label{verteks-ec} \eea
in the usual sense.
But we believe that the certain generalised vertex operator of type  (\ref{verteks-ec}) can be constructed using irregular vertex operators as in the Gaiotto paper \cite{Gaiotto}.
\end{remark}

\subsection{Free field realisation of $L(d,\bm{\lambda},\bm{\mu})$: the case $M_{1}(\bm{\lambda},\bm{\mu})$ is a $\Pi(0)$--module. }

We start with realisation of $M_{1}(\bm{\lambda},\bm{\mu})$ as a $\Pi(0)$--module. Note  that the operator $\delta(0) \in \mbox{End} (M_{1}(\bm{\lambda},\bm{\mu}))$ commutes with the action of  $M_{\varphi}(0) \otimes M_{\gamma, \delta}(1)$, and since
$M^0 = \mathcal W_{1+ \infty,-1} \subset M_{\varphi}(0) \otimes M_{\gamma, \delta}(1)$, we get that $\delta(0)$ commutes with the action of  $\mathcal W_{1+ \infty,-1}$.

By construction we have that $M_{1}(\bm{\lambda},\bm{\mu})$ is a module for  the Heisenberg vertex algebra $ M_{\gamma, \overline \delta }(1)$, generated by $\gamma (z) $ and $\overline \delta (z)  = \delta(z)  -2  \varphi (z)$, with the free action of $\overline \delta(0) = \delta(0) -2  \chi_0$.

For each $d \in {\C}$,  we get the following irreducible  $\mathcal W_{1+ \infty,-1}$--module $$  L(d,\bm{\lambda},\bm{\mu}) =\frac{ M_{1}(\bm{\lambda},\bm{\mu})}{ J_d},$$
where $J_d   =  \mathcal W_{1+ \infty,-1}.  (\overline \delta(0) -d) w_{\lambda, \chi} $.

This implies that $   L(d,\bm{\lambda},\bm{\mu})$ is an $ M_{\gamma, \overline \delta }(1)$--module such that $\overline \delta(0)$ acts as $d \cdot \mbox{Id}$. Therefore we have the following:

\begin{proposition} \label{identification-1}
 Identify $\mathcal W_{1+ \infty,-1}$ as a subalgebra of the Heisenberg vertex algebra
 $M_{\gamma, \overline \delta }(1)$.    Assume that
 $$\bm{\lambda}=(\lambda, 0, 0, \dots ), \ \bm{\mu} = (\mu_1, \cdots, \mu_n). $$
 Then the Whittaker $\mathcal W_{1+ \infty,-1}$--module $ L(d,\bm{\lambda},\bm{\mu})$  has the structure of an   irreducible  Whittaker  module for the Hiesenberg vertex algebra $M_{\gamma, \overline \delta }(1)$, generated by the Whittaker vector $\wh$
  such that 
 $ \gamma(0) \equiv - \mbox{Id}, \quad \overline \delta(0)  \equiv  d \cdot  \mbox{Id}$ and $$ \gamma(i) {\wh} =0, \quad   \overline \delta(n) {\wh} =   - 2 \frac{\mu_i}{\lambda}  {\wh} \quad (n \in {\Z}_{>0}). $$
\end{proposition}

 \subsection{  The general case  and non-tensor product modules}
Proposition  \ref{identification-1} shows that if $\bm{\lambda}=(\lambda, 0, 0, \dots )$, then 
$L(d,\bm{\lambda},\bm{\mu})$  has the structure of a module for the Heisenberg vertex algebra $M_2(1)$. One can ask if $L(d,\bm{\lambda},\bm{\mu})$ in general  are  isomorphic to a Whittaker module for $M_2(1)$. 

It is important note that  $M^0= W_{1+\infty, c=-1} \cong \mathcal M(2) \otimes M_1(1)$, and therefore Whittaker modules for $M^0$ can be  obtained by using recent Tanabe paper \cite{T2} on Whittaker modules for the singlet algebra $\mathcal M(p)$, $p \ge 2$.  He proved in \cite[Theorem 1.1]{T2} that every irreducible $\mathcal M(p)$--module which is generated by the Virasoro Whittaker vector can be realized as a Whittaker  module for the Heisenberg vertex algebra $M_1(1)$.

 Next result shows that in general $L(d,\bm{\lambda},\bm{\mu})$  is not realized as an irreducible, Whittaker  $M_2(1)$--module.
 

 
 \begin{proposition} \label{non-isom} Assume that $\bm{\lambda} = (\lambda_0, \dots, \lambda_n)$, $\bf{\mu}  = (\mu_1, \dots, \mu_m)$,  $n >0$ and  $\lambda_n, \mu_m \ne 0$. 
Then $L(d,\bm{\lambda}, \bm{\mu}) $ does not have the form (\ref{form-decomp}), i.e, it is not realized as a Whittaker $M_2(1)$--module.
 %
 \end{proposition} 
 
\begin{proof}
  
 For any $v \in M_1 (\bm{\lambda},\bm{\mu})$, denote by $[v]$ the projection of $v$ to $L(d, \bm{\lambda},\bm{\mu})$. Then $[\wh]$ is the Whittaker vector $L(d, \bm{\lambda},\bm{\mu})$ for the Whittaker pair $(\widehat{\mathfrak{gl}}, \mathfrak p)$ (cf. Subsection \ref{gl-pair}),  which is unique, up to a scalar factor.
 
  Assume that $L(d,\bm{\lambda}, \bm{\mu})\cong Z_1 \otimes Z_2$ where $Z_1$ is  an irreducible  Whittaker $\mathcal M(2)$--module, and $Z_2$ is an irreducible Whittaker $M_{J^0}(1)$--module. An irreducible Whittaker $M_{J^0}(1)$--module is generated by the Whittaker vector $wh_{heis}$ such that for $k \ge 0$:
  $$ J^0( k) wh_{heis} = \chi_k wh_{heis}, \quad \ch_k \in {\C}, $$
  and one easily sees that each vector in the Whittaker module  must be  locally finite for $J^0(k)$, $k \ge 0$.
This implies that   arbitrary vector $w \in L(d,\bm{\lambda}, \bm{\mu})$ is locally finite for $J^0(k_0)$  for all  $k_0 >0$, meaning that
\bea  \dim \mbox{span}_{\C} \{ J ^0( k_0) ^i w \vert \ i \ge 0 \} < \infty. \label{konacno} \eea
  We shall prove that it is not possible. We shall find $k_0 >0$  and $w$ such that (\ref{konacno}) does not hold. We take  $w = [\wh]$.

  Now we consider
  $J^0(k_0) ^i \wh$, $i \ge 1$. For simplicity we consider the case $n \ge m$, when $k_0 =n$. We have 
  $$ \mbox{span}_{\C} \{ J ^0( k_0) ^i [\wh]  \vert \ i \ge 0 \}  = \mbox{span}_{\C} 
  \{ [a^*(0) ^i \wh]  \vert \ i \ge 0 \}.$$

  Assume that  there are constants $C _0, \dots, C_m$, $C_m \ne 0$ such that
  \bea  C_0 [\wh] + \cdots + C_m  [a^*(0) ^m  \wh] = 0. \label{lin-comb} \eea
Set
$$ u= C_0 \wh + \cdots + C_m  a^*(0) ^m  \wh \in M_{1}(\bm{\lambda}, \bm{\mu}). $$

 Since 
$$\frac{1}{\lambda_n} (E_{0, m} -\lambda_0 \mu_m ) ^m u    = \nu \wh \quad (\nu \ne 0),$$
we conclude that $u$ is not in  the ideal generated by $(I-d) {\C}[I]$, implying that $[u] \ne 0$.  This contradicts relation (\ref{lin-comb}). Therefore 
$ \mbox{span}_{\C} \{ J ^0( k_0) ^i [\wh]  \vert \ i \ge 0 \} $ is infinite-dimensional, which implies that $L(d, \bm{\lambda}, \bm{\mu})$ does not have the form  (\ref{form-decomp}). The proof follows.

   \end{proof}

  \begin{remark} In Appendix we shall present a different proof of Proposition \ref{non-isom}  by proving that  $L(d,\bm{\lambda}, \bm{\mu}) $ is isomorphic an  irreducible module for the Heisenberg-Virasoro algebra which is not a tensor product module.
  \end{remark}
 
\section{Generalized Whittaker modules for  $\g = \mathfrak{gl}(2\ell , {\C})$}
\label{sect-gl2n}
 
In this Section we take $\g = \mathfrak{gl}(2\ell , {\C})$. Recall from Example \ref{ex-gl2n} that $\g$ has triangular decomposition $\g = \mathfrak n_- \oplus \mathfrak h \oplus \mathfrak n_+$. We take the subalgebra $\mathfrak n \subset \mathfrak n_+$.
$$ \mathfrak n = \mbox{span}_{\C} \{ e_{i, j + \ell} \ \vert  i, j = 1, \dots, \ell \}. $$
 Now we are interested in Whittaker modules 
 with respect to the pair $(\g, \mathfrak n)$, which we   called generalized Whittaker modules.  In particular we are interested in a construction of simple quotients of the universal Whittaker module $$ M_{\lambda} = U(\mathfrak g) \otimes _{ U(\mathfrak n)} {\C} v_{\lambda}.$$

 \subsection{Weyl algebra $\mathcal A_{2\ell}$ and its Whittaker modules.} Here we recall the definition of the  Weyl algebra $\mathcal A_{\ell}$. It is an complex  associative algebra with generators
 $$ a_i, a^* _i, \quad i=1, \dots, 2\ell$$
 and relations
 $$ [a_i, a_j] = [a^* _i, a^*_j] =0, \quad [a_i, a^*_j] = \delta_{i,j}, \quad (i,j=1, \dots, 2\ell). $$
 
 Define the normal ordering on $\mathcal A$ with
 $$ :x y := \frac{1}{2} ( x y + y x). $$
 In paricular, we have
 $$ :a_i a^* _i: = a_i a^* _i - \frac{1}{2} =  a^* _i a _i  +   \frac{1}{2}. $$
 
 For $\bm{\alpha} = ( \alpha_1,  \dotsc, \alpha_{\ell})$, $\bm{\beta} = (\beta_1, \dotsc, \beta_{\ell} ) \in {\C} ^{\ell}$, we define the Whittaker module $W(\alpha, \beta)$ to be the quotient
\[W(\bm{\alpha}, \bm{\beta}) = \sfrac{ {\mathcal{A}_{2 \ell}}}{{I_{\ell}(\bm{\alpha}, \bm{\beta})}},\]
 where    and $I_{\ell}(\bm{\alpha}, \bm{\beta}) $ is the left ideal \[{I_{\ell}(\bm{\alpha}, \bm{\beta})} = \big\langle a_1-\alpha_1, \dotsc, a_{\ell}  - \alpha_{\ell}, a^{*}_{\ell +1} -\beta_{1} , \dotsc, a^{*}_{2 \ell}  - \beta_{\ell}  \big\rangle.\]

One shows that

\begin{lemma}  $W(  \bm{ \alpha}, \bm{ \beta}) $ is an irreducible ${\mathcal{A}_{2 \ell}}$--module. It is generated by the Whittaker vector $\whf = 1 + {I_{\ell}(\bm{\alpha}, \bm{\beta})}$ such that
$$ a_i \whf  = \alpha_i \whf,  \quad  a^* _{\alpha +i} \whf  = \beta_i \whf, \quad i=1, \dots, \ell.$$
As a vector space 
$$ W(  \bm{ \alpha}, \bm{ \beta}) \cong {\C} [a^* _1, \cdots, a^* _{\ell}, a_{\ell +1}, \cdots, a_{2 \ell}]. $$
\end{lemma}

\subsection{ Embedding of $\g$ into $\mathcal A_{2 \ell}$.}
It is well known that there  is a Lie algebra homomorphism
$$ \Phi : \g \rightarrow  \left( \mathcal A_{2 \ell}\right)_{Lie}$$ uniquely determined by
$$ e_{i,j} \mapsto  : a_i a^* _j:. $$
This implies that each $\mathcal A_{2 \ell}$--module can be treated as  a $\g$--module.

 \subsection{$W(  \bm{ \alpha}, \bm{ \beta}) $ as a $\g$--modules.}
 Now we consider the generalized Whittaker modules $W(  \bm{ \alpha}, \bm{ \beta}) $ as $\g$--modules.  

\begin{proposition} \label{cyclic-gl2n} Assume that $\bm{\alpha} \ne 0$  and  $\bm{\beta} \ne0$. $W(  \bm{ \alpha}, \bm{ \beta}) $  is a Whittaker module for the Whittaker pair $(\g, \n)$. In a particular, $\whf$ is a Whittaker vector such that
$$ e_{i, \ell + j} \whf = \alpha_i \beta_j  \whf \quad \forall i,j \in \{1, \dots, \ell\}$$ and
$ W(  \bm{ \alpha}, \bm{ \beta}) = U(\g) \whf. $

In particular, $W(  \bm{ \alpha}, \bm{ \beta}) $  is a quotient of the universal Whittaker module $M_{\lambda}$ with Whittaker function:
$$ \lambda= \lambda_{  \bm{ \alpha}, \bm{ \beta}} :  {\n}  \rightarrow {\C}, \quad \lambda(e_{i, j+ \ell}) = \alpha_i \beta_j,  \quad  \forall i,j \in \{1, \dots, \ell\}. $$
\end{proposition}
\begin{proof}
The proof is completely analogous to that presented in Proposition \ref{ired-gl} below in the case $\widehat{\mathfrak{gl}}$.
\end{proof}

Note that the central  element $ \sum_{i=1} ^{2 \ell}e_{i,i}  \in \g$ acts on $W(  \bm{ \alpha}, \bm{ \beta}) $ as
$$I =  \sum_{i=1} ^{2 \ell} : a_i a^* _i:.
 $$
 
 Since $$I \whf =\sum_{i=1} ^{\ell} ( \alpha_i a^* _{i}  + \beta_i a_{i+\ell}) \whf, $$ which is not proportional to $\whf$, we conclude that $W(  \bm{ \alpha}, \bm{ \beta}) $ is reducible $\g$--module.

\begin{theorem} For each $d \in {\C}$, $\bm{\alpha},\bm{\beta} \in {\C}^{\ell}$,  $\bm{\alpha},\bm{\beta} \ne 0$, we have
$$L(d,\bm{\alpha},\bm{\beta}) = \frac{   W(\bm{\alpha},\bm{\beta}) }  {\langle  (I-d) {\C}[I] {\whf} \rangle  }$$
is an irreducible $\g$--module.
\end{theorem}
\begin{proof}Proof is analogous to that of  Theorem \ref{irred-description}.
\end{proof}

 \section*{Appendix: $M_1(\bm{\lambda}, \bm{\mu})$ as a module for the Heisenberg-Virasoro algebra}
 
 In this section we prove  a stronger statement that $M_1(\bm{\lambda}, \bm{\mu})$ is a cyclic module for the Heisenberg-Virasoro vertex subalgebra of $\mathcal W_{1+\infty,-1}$. Our main argument will be that  the action of singlet $\mathcal M(2)$ on the Whittaker vector $\wh$ can be replaced by the action of the singlet Virasoro subalgebra. Next we identify $M_1(\bm{\lambda}, \bm{\mu})$ and $L(d,\bm{\lambda}, \bm{\mu}) $  as restricted modules for the Heisenberg-Virasoro algebra recently classified in \cite{TYZ}.
 \vskip 5mm 
 
 Recall that
 $$J^k(z)  =Y(a^*(-k) a, z) = \sum _{ s \in {\Z} } J^k(s) z^{-k-s-1}. $$
 
 The singlet algebra $\mathcal M(2)$ (cf. \cite{A-singlet}, \cite{Wa}) is generated by the Virasoro vector
 $L = J^1 + \tfrac{1}{2} :(J^0)^2: + \tfrac{1}{2} D J^0$ and by the primary field $H$ of conformal weight $3$.  Recall that the ${\Z}_2$--orbifold of $M$ is isomorphic to the simple affine vertex algebra $L_{-1/2}(\mathfrak{sl}_2)$ generated by
 $$ e = \tfrac{1}{2} : a^2:, h = -J^0, f =-\tfrac{1}{2} :(a^* )^2:,$$ the singlet algebra is isomorphic to the parafermion vertex algebra  $N_{-1/2}(\mathfrak{sl}_2)$ and  $H$ is a scalar multiple of the parafermion  generator   $W^3$ in \cite[Section 2]{DLY}:
\bea    W^3&=&  k^2 h(-3){\bf 1}  + 3k h(-2) h(-1){\bf 1}  + 2 h(-1) ^3{\bf 1}  - 6 k h(-1) e(-1) f(-1){\bf 1} \nonumber \\
&&  + 3 k^2 e(-2) f(-1) {\bf 1} - 3 k^2 f(-2) e(-1)  {\bf 1}  \quad (k=-1/2),\nonumber \eea
 (see also  \cite[Section 4]{R10}).

We have that  $\mathcal L^{HVir} _{c=-1} = \mathcal L^{Vir} _{c=-2} \otimes M_1(1) $, where  $\mathcal L^{Vir} _{c=-2}$ denotes the simple Virasoro vertex algebra of central charge $c=-2$ generated by $L$.

 \begin{theorem} \label{app} We have:
 \item[(1)]$M_1(\bm{\lambda}, \bm{\mu})$ is an cyclic  $\mathcal L^{HVir} _{c=-1}$--module. 
 \item[(2)]  $L(d,\bm{\lambda}, \bm{\mu}) $ is an irreducible $\mathcal L^{HVir}_{c=-1}$--module.
 \item[(3)] $M_1(\bm{\lambda}, \bm{\mu})$ and $L(d,\bm{\lambda}, \bm{\mu}) $ are not tensor product  $\mathcal L^{HVir}_{c=-1}$--modules. Moreover, there is a nilpotent subalgebra $\mathcal P$ of $\mathcal H$  and $1$--dimensional $\mathcal P$--module $U$ such that $M_1(\bm{\lambda}, \bm{\mu}) = \mbox{Ind} _{\mathcal P    } ^{\mathcal H} U$.
 \end{theorem}
 
 \begin{proof}
 The field $H$ can be expressed using $J^0, J^1$ and it contains the non-trivial summand $J^0(-1)^3{\bf 1}$. Set $H(i) := H_{i+2}$, so we have $$H(z) =  Y(H, z) = \sum_{i \in {\Z} } H(i) z^{-i-3}. $$ 
 By a direct calculation we get for $s \in {\Z}_{\ge 0}$: 
 $$H(3n + 3m+s ) \wh =  \delta_{s,0} q \wh $$
 for certain $q \in {\C}$, $  q \ne 0$.

  Using the following relation in $\mathcal M(2)$ 
  $$  \frac{3}{4} H(-6){\bf 1}  - L(-2)   H(-4){\bf 1}  +\frac{3}{2} L(-3) H  =0,$$
 and using the  arguments as in  \cite[Lemma 3.1]{T2} we get that
 $$ \mathcal M(2). \wh = \langle L \rangle \wh, $$
 so $\mathcal M(2). \wh$ is isomorphic to the Virasoro submodule obtained by the action of the Virasoro generator $L$ on  the Whittaker vector $\wh$.
 Since  by Theorem \ref{structure-equivalence} $M_1(\bm{\lambda}, \bm{\mu})$ is a cyclic $\mathcal M(2) \otimes M_1(1)$--module, we conclude that $M_1(\bm{\lambda}, \bm{\mu})$ is a cyclic $\mathcal L^{HVir} _{c=-1}$--module. This proves assertion (1). The assertion (2) is a consequence of (1).

Let us prove assertion (3).
For   simplicity we consider the case $n + 1 \ge m$. Other cases can be  treat similarly.

 By a direct calculation we get for $s \in {\Z}_{>0}$:
 \bea  && J^0(n+m +s) \wh = J^1(n+m +s) \wh = 0,  \label{rel-0} \\
  && J^0( n+m) \wh = a_{m} \wh, \dots, J^0(n+1) \wh = a_{1} \wh, \label{rel-1} \\
 && J^1( n+m) \wh= b_{m+1} \wh, \dots, J^1 (n) \wh = b_1 \wh, \label{rel-2} \eea
 where $a_i, b _i  \in  {\C}$, $a_{m}, b_{ m+1} \ne 0$,  and
 $J^0(n) \wh$, $J^1(n-1) \wh$ are not proportional to $\wh$.

   Consider the nilpotent  subalgebra $\mathcal P^{(n)}$ of $\mathcal H$ spanned by $J^0(r+1), J^1(r)$, $r \ge n$ and $C_1, C_2$.  Then ${\C}\wh$ is the $1$-dimensional $\mathcal P^{(n)}  $--module satisfying relations (\ref{rel-0})-(\ref{rel-2}) and 
 $C_1 \wh = 2 \wh$, $C_2 \wh = -\wh$.
   Then $M_1(\bm{\lambda}, \bm{\mu})$ and    $L(d, \bm{\lambda}, \bm{\mu})$ are quotients of the universal $\mathcal H$--module:
   $$ \mbox{Ind} _{\mathcal P^{(n)}  }^{\mathcal H} {\C} \wh. $$
   These universal modules and their simple quotients  appeared  in a slightly general  form in \cite[Section 7]{TYZ}, where it was proved that they are not tensor product modules (see \cite[Example  7.5, Remark A.4]{TYZ}). In particular, it was shown that for each $d \in {\C}$: 
  $$ \frac{ \mbox{Ind} _{\mathcal P^{(n)}  }^{\mathcal H} } {U(\mathcal H)(J^0(0)-d )\wh}$$
  is a simple, restricted $\mathcal H$--module. This easily implies that 
  as   $\mathcal H$--modules:
  $$M_1(\bm{\lambda}, \bm{\mu}) = \mbox{Ind} _{\mathcal P^{(n)}   }^{\mathcal H} {\C} \wh, 
   \quad  L(d, \bm{\lambda}, \bm{\mu}) = \frac{M_1(\bm{\lambda}, \bm{\mu}) }{U(\mathcal H)(J^0(0)-d )\wh}.$$
   The proof follows.
 \end{proof}

\begin{corollary}   $\mathcal W_{1+\infty, c=-1}$--modules $\widetilde \rho_s (L(d,\bm{\lambda}, \bm{\mu}))$ are typical.
\end{corollary}

\vskip10pt {\footnotesize{}{ }\textbf{\footnotesize{}D.A.}{\footnotesize{}:
Department of Mathematics, Faculty of Science, University of Zagreb, Bijeni\v{c}ka 30,
10 000 Zagreb, Croatia; }\texttt{\footnotesize{}adamovic@math.hr}{\footnotesize \par}


\textbf{\footnotesize{}V.P.}{\footnotesize{}: Department of Mathematics, Faculty of Science,
University of Zagreb, Bijeni\v{c}ka 30, 10 000 Zagreb, Croatia; }\texttt{\footnotesize{}vpedic@math.hr}{\footnotesize \par}



\end{document}